\documentclass[12pt]{article}%
\usepackage{epsf}
\usepackage{amssymb}
\usepackage{amsmath}
\usepackage{amsfonts}
\usepackage{graphicx}%
\newtheorem{theorem}{Theorem}[section]
\newtheorem{acknowledgement}[theorem]{Acknowledgement}
\newtheorem{algorithm}[theorem]{Algorithm}

\newtheorem{conjecture}[theorem]{Conjecture}
\newtheorem{corollary}[theorem]{Corollary}

\newtheorem{definition}[theorem]{Definition}

\newtheorem{lemma}[theorem]{Lemma}

\newtheorem{proposition}[theorem]{Proposition}

\newenvironment{proof}[1][Proof]{\noindent\textbf{#1.} }{\ \rule{0.5em}{0.5em}}
\begin{document}

\author{Yulia Kempner $^{1}$\\$^{1}$ Holon Institute of Technology, ISRAEL\\yuliak@hit.ac.il
\and Vadim E. Levit $^{1,2}$\\$^{2}$ Ariel University Center of Samaria, ISRAEL\\levitv@ariel.ac.il}
\title{Geometry of antimatroidal point sets}
\date{}
\maketitle

\begin{abstract}
The notion of "antimatroid with repetition" was conceived by Bjorner, Lovasz
and Shor in 1991 as a multiset extension of the notion of antimatroid
\cite{ChipF}. When the underlying set consists of only two elements, such
two-dimensional antimatroids correspond to point sets in the plane. In this
research we concentrate on geometrical properties of antimatroidal point sets
in the plane and prove that these sets are exactly parallelogram polyominoes.
Our results imply that two-dimensional antimatroids have convex dimension $2$.
The second part of the research is devoted to geometrical properties of
three-dimensional antimatroids closed under intersection.

\bigskip\textbf{Keywords: } antimatroid, convex dimension, lattice animal, polyomino.

\end{abstract}

\section{Preliminaries}

An antimatroid is an accessible set system closed under union \cite{BZ}. An
algorithmic characterization of antimatroids based on the language definition
was introduced in \cite{BF}. Another algorithmic characterization of
antimatroids that depicted them as set systems was developed in \cite{KL}.
While classical examples of antimatroids connect them with posets, chordal
graphs, convex geometries, etc., game theory gives a framework in which
antimatroids are interpreted as permission structures for coalitions
\cite{BJ}. There are also rich connections between antimatroids and cluster
analysis \cite{Malta}. In mathematical psychology, antimatroids are used to
describe feasible states of knowledge of a human learner \cite{Eppstein}.

In this paper we investigate the correspondence between antimatroids and
polyominoes. In the digital plane ${\mathbb{Z}}^{2}$, \emph{a polyomino}
\cite{Golomb} is a finite connected union of unit squares without cut points.
If we replace each unit square of a polyomino by a vertex at its center, we
obtain an equivalent object named \emph{a lattice animal} \cite{LA}. Further,
we use the name polyomino for the two equivalent objects.

A polyomino is called column-convex (resp. row-convex) if all its columns
(resp. rows) are connected (in other words, each column/row has no holes). A
convex polyomino is both row- and column-convex. The parallelogram polyominoes
\cite{DeLungo,DV,DDD}, sometimes known as staircase polygons \cite{BM,NB}, are
a particular case of this family. They are defined by a pair of monotone paths
made only with north and east steps, such that the paths are disjoint, except
at their common ending points.

In this paper we prove that antimatroidal point sets in the plane and
parallelogram polyominoes are equivalent.

Let $E$ be a finite set. A \textit{set system }over $E$ is a pair
$(E,\mathcal{F})$, where $\mathcal{F}$ is a family of sets over $E$, called
\textit{feasible }sets. We will use $X\cup x$ for $X\cup\{x\}$, and $X-x$ for
$X-\{x\}$.

\begin{definition}
\cite{Greedoids}A finite non-empty set system $(E,\mathcal{F})$ is an
\textit{antimatroid }if

$(A1)$ for each non-empty $X\in\mathcal{F}$, there exists $x\in X$ such that
$X-x\in\mathcal{F}$

$(A2)$ for all $X,Y\in\mathcal{F}$, and $X\nsubseteq Y$, there exists $x\in
X-Y$ such that $Y\cup x\in\mathcal{F}$.
\end{definition}

\smallskip Any set system satisfying $(A1)$ is called \textit{accessible.}

In addition, we use the following characterization of antimatroids.

\begin{proposition}
\cite{Greedoids} For an accessible set system $(E,\mathcal{F})$ the following
statements are equivalent:

$(i)$ $(E,\mathcal{F})$ is an antimatroid

$(ii)$ $\mathcal{F}$ is closed under union ($X,Y\in\mathcal{F}\Rightarrow
X\cup Y\in\mathcal{F}$)
\end{proposition}

Consider another property of antimatroids.

\begin{definition}
A set system $(E,\mathcal{F})$ satisfies the \textit{chain property }if for
all $X,Y\in\mathcal{F}$, and $X\subset Y$, there exists a chain $X=X_{0}%
\subset X_{1}\subset...\subset X_{k}=Y$ such that $X_{i}=X_{i-1}\cup
x_{i\text{ }}$and $X_{i}\in\mathcal{F}$ for $0\leq i\leq k$.
\end{definition}

It is easy to see that the chain property follows from $(A2)$, but these
properties are not equivalent.

A poly-antimatroid \cite{NM} is a generalization of the notion of the
antimatroid to multisets. \emph{A poly-antimatroid} is a finite non-empty
multiset system $(E,S)$ that satisfies the antimatroid properties $(A1)$ and
$(A2)$.

Let $E=\{x,y\}$. In this case each point $A=(x_{A},y_{A})$ in the digital
plane ${\mathbb{Z}}^{2}$ may be considered as a multiset $A$ over $E$, where
$x_{A}$ is a number of repetitions of an element $x$, and $y_{A}$ is a number
of repetitions of an element $y$ in multiset $A$. Consider a set of points in
the digital plane ${\mathbb{Z}}^{2}$ that satisfies the properties of an
antimatroid. That is a two-dimensional poly-antimatroid.

\begin{definition}
\label{Def} A set of points $S$ in the digital plane ${\mathbb{Z}}^{2}$ is an
antimatroidal point set if

$(A1)$ for every point $(x_{A},y_{A})\in S$, such that $(x_{A},y_{A}%
)\neq(0,0)$,

either $(x_{A}-1,y_{A})\in S$ or $(x_{A},y_{A}-1)\in S$

$(A2)$ for all $A\nsubseteq B\in S$,

if $x_{A}\geq x_{B}$ and $y_{A}\geq y_{B}$ then either $(x_{B}+1,y_{B})\in S$
or $(x_{B},y_{B}+1)\in S$

if $x_{A}\leq x_{B}$ and $y_{A}\geq y_{B}$ then $(x_{B},y_{B}+1)\in S$

if $x_{A}\geq x_{B}$ and $y_{A}\leq y_{B}$ then $(x_{B}+1,y_{B})\in S$
\end{definition}

Accessibility implies that $\emptyset\in S$.

For example, see an antimatroidal point set in Figure \ref{int_fig1}.

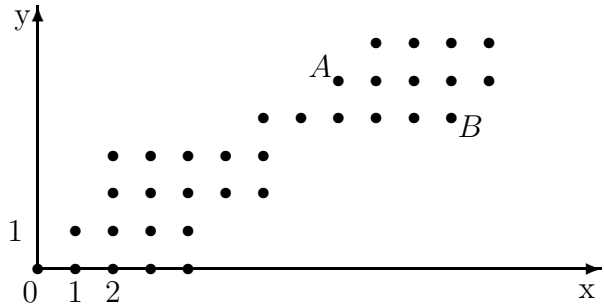
\begin{figure}[ptbh]
\setlength{\unitlength}{1cm}
\par
\begin{picture}(8,4.5)(-3,-0.5)\thicklines
\put(0,0){\vector(1,0){7.5}} \put(0,0){\vector(0,1){3.5}}
\multiput(0,0)(0.5,0){5}{\circle*{0.15}}
\multiput(0.5,0.5)(0.5,0){4}{\circle*{0.15}}
\multiput(1,1)(0.5,0){5}{\circle*{0.15}}
\multiput(1,1.5)(0.5,0){5}{\circle*{0.15}}
\multiput(3,2)(0.5,0){6}{\circle*{0.15}}
\multiput(4,2.5)(0.5,0){5}{\circle*{0.15}}
\multiput(4.5,3)(0.5,0){4}{\circle*{0.15}}
\put(7.3,-0.3){\makebox(0,0){x}} \put(-0.2,3.3){\makebox(0,0){y}}
\put(-0.1,-0.3){\makebox(0,0){0}} \put(0.5,-0.3){\makebox(0,0){1}}
\put(1,-0.3){\makebox(0,0){2}} \put(-0.3,0.5){\makebox(0,0){1}}
\put(3.75,2.7){\makebox(0,0){$A$}}
\put(5.75,1.9){\makebox(0,0){$B$}}
\end{picture}
\caption{An antimatroidal point set.}%
\label{int_fig1}%
\end{figure}

A three-dimensional point set is defined similarly to a two-dimensional set.

\section{Two-dimensional antimatroidal point sets and polyominoes}

In this section we consider a geometric characterization of two-dimensional
antimatroidal point sets. The following notation \cite{Digital} is used. If
$A=(x,y)$ is a point in a digital plane, the \textit{4-neighborhood}
$N_{4}(x,y)$ is the set of points
\[
N_{4}(x,y)=\{(x-1,y),(x,y-1),(x+1,y),(x,y+1)\}
\]
and \textit{8-neighborhood} $N_{8}(x,y)$ is the set of points
\[
N_{8}(x,y)=\{(x-1,y),(x,y-1),(x+1,y),(x,y+1),(x-1,y-1),
\]%
\[
(x-1,y+1),(x+1,y-1),(x+1,y+1)\}.
\]
Let $m$ be any of the numbers $4$ or $8$. A sequence $A_{0},A_{1},...,A_{n}$
is called an $N_{m}$-\textit{path }if $A_{i}\in N_{m}(A_{i-1})$ for each
$i=1,2,...n$. Any two points $A,B\in S$ are said to be $N_{m}$-connected in
$S$ if there exists an $N_{m}$-path $A=A_{0},A_{1},...,A_{n}=B$ from $A$ to
$B$ such that $A_{i}\in S$ for each $i=1,2,...n$.\ A digital set $S$ is an
$N_{m}$-\textit{connected} set if any two points $P$,$Q$ from $S$ are $N_{m}%
$-connected in $S$. An $N_{m}$-\textit{connected component} of a set $S$ is a
maximal subset of $S$, which is $N_{m}$-connected.

An $N_{m}$-path $A=A_{0},A_{1},...,A_{n}=B$ from $A$ to $B$ is called
\textit{a monotone increasing} $N_{m}$-path if $A_{i}\subset A_{i+1}$ for all
$0\leq i<n$, i.e.,
\[
(x_{A_{i}}<x_{A_{i+1}})\wedge(y_{A_{i}}\leq y_{_{A_{i+1}}})\text{ or
}(x_{_{A_{i}}}\leq x_{_{A_{i+1}}})\wedge(y_{_{A_{i}}}<y_{_{A_{i+1}}}).
\]

The chain property and the fact that the family of feasible sets
of an antimatroid is closed under union mean that for each two
points $A,B$: if $B\subset A$, then there is a monotone increasing
$N_{4}$-path from $B$ to $A$, and if $A$ is non-comparable with
$B$, then there is a monotone increasing $N_{4}$-path from both
$A$ and $B$ to $A\cup B=(\max(x_{A},x_{B}),\max(y_{A},y_{B}))$. In
particular, for each $A\in S$ there is a monotone decreasing
$N_{4}$-path from $A$ to $0$. So, we can conclude that an
antimatroidal point set is an $N_{4}$-connected component in the
digital plane ${\mathbb{Z}}^{2}$.

\begin{definition}
A point set $S\subseteq{\mathbb{Z}}^{2}$ is defined to be orthogonally convex
if, for every line L that is parallel to the x-axis ($y=y^{\ast}$) or to the
y-axis ($x=x^{\ast}$), the intersection of $S$ with $L$ is empty, a point, or
a single interval ($[(x_{1},y^{\ast}),(x_{2},y^{\ast})]=\{(x_{1},y^{\ast
}),(x_{1}+1,y^{\ast}),...,(x_{2},y^{\ast})\}$).
\end{definition}

It follows immediately from the chain property that any antimatroidal point
set $S$ is an orthogonally convex connected component\textit{.}

Thus, antimatroidal point sets are convex polyominoes.

In the following we prove that antimatroidal point sets in the plane closed
not only under union, but under intersection as well.

\begin{lemma}
\label{Square_P} An antimatroidal point set in the plane is closed under
intersection, i.e., if two points $A=(x_{A},y_{A})$ and $B=(x_{B},y_{B})$
belong to an antimatroidal point set $S$, then the point $A\cap B=(min(x_{A}%
,x_{B}),min(y_{A},y_{B}))\in S$.
\end{lemma}

\begin{proof}
The proposition is evident for two comparable points. Consider two
non-comparable points $A$ and $B$, and assume without loss of generality that
$x_{A}<x_{B}$ and $y_{A}>y_{B}$. Then there is a monotone decreasing $N_{4}%
$-path from $A$ to $0$, and so there is a point $C=(x_{C},y_{B})\in S$ on this
path with $x_{C}\leq x_{A}$. Hence, the point $A\cap B$ belongs to $S$, since
it is located on the monotone increasing $N_{4}$-path from $C$ to $B$.
\end{proof}

Lemma \ref{Square_P} implies that every antimatroidal point set is a union of
rectangles built on each pair of non-comparable points. The following theorem
shows that antimatroidal point sets in the plane and parallelogram polyominoes
are equivalent.

\begin{theorem}
\label{Antimatroid_Set Theorem} A set of points $S$ in the digital plane
${\mathbb{Z}}^{2}$ is an antimatroidal point set if and only if it is an
orthogonally convex $N_{4}$-connected set that is bounded by two monotone
increasing $N_{4}$-paths between $(0,0)$ and $(x_{max},y_{max})$.
\end{theorem}

To prove the "if" part of Theorem \ref{Antimatroid_Set Theorem} it remains to
demonstrate that every antimatroidal point set $S$ is bounded by two monotone
increasing $N_{4}$-paths between $(0,0)$ and $(x_{max},y_{max})$. To give a
definition of the boundary we begin with the following notions.

A point $A$ in set $S$ is called an\textit{\ interior point} in $S$ if
$N_{8}(A)\in S$. A point in $S$ which is not an interior point is called a
\textit{boundary point}. All boundary points of $S$ constitute the boundary of
$S$. We can see an antimatroidal point set with its boundary in Figure
\ref{int_fig2}.

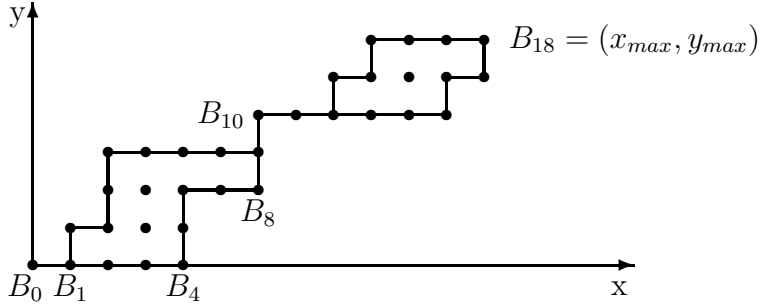
\begin{figure}[ptbh]
\setlength{\unitlength}{1cm}
\par
\begin{picture}(8,4.5)(-3,-0.5)\thicklines
\put(0,0){\vector(1,0){8}} \put(0,0){\vector(0,1){3.5}}
\multiput(0,0)(0.5,0){5}{\circle*{0.15}}
\multiput(0.5,0.5)(0.5,0){4}{\circle*{0.15}}
\multiput(1,1)(0.5,0){5}{\circle*{0.15}}
\multiput(1,1.5)(0.5,0){5}{\circle*{0.15}}
\multiput(3,2)(0.5,0){6}{\circle*{0.15}}
\multiput(4,2.5)(0.5,0){5}{\circle*{0.15}}
\multiput(4.5,3)(0.5,0){4}{\circle*{0.15}}
\put(3,2){\line(1,0){2.5}} \put(0.5,0.5){\line(1,0){0.5}}
\put(1,1.5){\line(1,0){2}} \put(4,2.5){\line(1,0){0.5}}
\put(4.5,3){\line(1,0){1.5}} \put(2,1){\line(1,0){1}}
\put(5.5,2.5){\line(1,0){0.5}} \put(0.5,0){\line(0,1){0.5}}
\put(1,0.5){\line(0,1){1}} \put(2,0){\line(0,1){1}}
\put(3,1){\line(0,1){1}} \put(4,2){\line(0,1){0.5}}
\put(5.5,2){\line(0,1){0.5}} \put(4.5,2.5){\line(0,1){0.5}}
\put(6,2.5){\line(0,1){0.5}} \put(7.8,-0.3){\makebox(0,0){x}}
\put(-0.2,3.3){\makebox(0,0){y}}
\put(-0.1,-0.3){\makebox(0,0){$B_{0}$}}
\put(0.5,-0.3){\makebox(0,0){$B_{1}$}}
\put(2,-0.3){\makebox(0,0){$B_{4}$}}
\put(3,0.7){\makebox(0,0){$B_{8}$}}
\put(8,3){\makebox(0,0){$B_{18}=(x_{max},y_{max})$}}
\put(2.5,2){\makebox(0,0){$B_{10}$}}
\end{picture}
\caption{A boundary of an antimatroidal point set.}%
\label{int_fig2}%
\end{figure}

Since antimatroidal sets in the plane are closed under union and under
intersection, there are six types of boundary points that we divide into two
sets -- lower and upper boundary:
\[
\mathcal{B}_{lower}=\{(x,y)\in S:(x+1,y)\notin S\vee(x,y-1)\notin
S\vee(x+1,y-1)\notin S\}
\]%
\[
\mathcal{B}_{upper}=\{(x,y)\in S:(x-1,y)\notin S\vee(x,y+1)\notin
S\vee(x-1,y+1)\notin S\}
\]

It is possible that $\mathcal{B}_{lower}\cap\mathcal{B}_{upper}\neq\emptyset$.
For example, the point $B_{10}$ in Figure \ref{int_fig2} belongs to both the
lower and upper boundaries.

\begin{lemma}
\label{Increase}The lower boundary is a monotone increasing path between
$(0,0)$ and $(x_{max},y_{max})$.
\end{lemma}

\begin{proof}
Let $B_{0}=(0,0),B_{1},...,B_{k}=(x_{max},y_{max})$ be a lexicographical order
of $\mathcal{B}_{lower}$. We will prove that it is a monotone increasing path.
Suppose the opposite. Then there is a pair of non-comparable points $B_{i}$
and $B_{j}$, such that $i<j$, $x_{B_{i}}<x_{B_{j}}$ and $y_{B_{i}}>y_{B_{j}}$.
Then Lemma \ref{Square_P} implies that there is a rectangle built on the pair
of non-comparable points. Hence, the point $B_{i}$ does not belong to the
lower boundary, which is a contradiction.
\end{proof}

\begin{lemma}
\label{Inside}For each $B_{k}=(x,y)\in\mathcal{B}_{lower}$ holds:

(i) if $A=(x,z)\in S$ and $A\notin\mathcal{B}_{lower}$ then $z>y$;

(ii) if $A=(z,y)\in S$ and $A\notin\mathcal{B}_{lower}$ then $z<x$.
\end{lemma}

\begin{proof}
Suppose the opposite, i.e., there is $A=(x,z)\in S$ such that $A\notin
\mathcal{B}_{lower}$ and $z<y$. Then the point $(x+1,z)\in S$, so there is the
last right point $(x^{\ast},z)\in S$, such that $x^{\ast}>x$. Hence, the point
$(x^{\ast},z)$ belongs to $\mathcal{B}_{lower}$, contradicting Lemma
\ref{Increase}. The Property (\textit{ii}) is proved similarly.
\end{proof}

\begin{lemma}
The lower boundary is an $N_{4}$-path between $(0,0)$ and $(x_{max},y_{max})$.
\end{lemma}

\begin{proof}
Prove that $|B_{i+1}-B_{i}|=1$ for every $1\leq i<k$. If $x_{B_{i}}%
=x_{B_{i+1}}$. then from the chain property and Lemma \ref{Inside} it
immediately follows that $y_{B_{i+1}}=y_{B_{i}}+1$.

If $x_{B_{i}}<x_{B_{i+1}}$, then the chain property implies that either
$(x_{B_{i}}+1,y_{B_{i}})\in S$ or $(x_{B_{i}},y_{B_{i}}+1)\in S$. Since the
point $(x_{B_{i}},y_{B_{i}}+1)\notin\mathcal{B}_{lower}$, we can conclude that
$(x_{B_{i}}+1,y_{B_{i}})\in S$ in any case. Now Lemma \ref{Increase} implies
that this point belongs to $\mathcal{B}_{lower}$, i.e., $B_{i+1}=(x_{B_{i}%
}+1,y_{B_{i}}).$
\end{proof}

The upper boundary case is validated in the same way.

Thus an antimatroidal point set is an orthogonally convex $N_{4}$-connected
set bounded by two monotone increasing $N_{4}$-paths.

The following lemma is the "only-if" part of Theorem
\ref{Antimatroid_Set Theorem}.

\begin{lemma}
An orthogonally convex $N_{4}$-connected set $S$ that is bounded by two
monotone increasing $N_{4}$-paths between $(0,0)$ and $(x_{max},y_{max})$ is
an antimatroidal point set.
\end{lemma}

\begin{proof}
By Definition \ref{Def} we have to check the two properties (A1) and (A2):

(A1) Let $A=(x,y)\in S$. If $A$ is an interior point in $S$ then $(x-1,y)\in
S$ and $(x,y-1)\in S$. If $A$ is a boundary point, then the previous point on
the boundary ($(x,y-1)$ or $(x-1,y)$) belongs to $S$.

(A2) Let $A\nsubseteq B\in S$. Consider two cases:

(i) $x_{a}\geq x_{b}$ and $y_{a}\geq y_{b}$. If $B$ is an interior point in
$S$ then $(x_{b}+1,y_{b})\in S$ and $(x_{b},y_{b}+1)\in S$. If $B$ is a
boundary point, then the next point on the boundary ($(x_{b},y_{b}+1)$ or
$(x_{b}+1,y_{b})$) belongs to $S$.

(ii)$x_{a}\leq x_{b}$ and $y_{a}\geq y_{b}$. We have to prove that
$(x_{b},y_{b}+1)\in S$. Suppose the opposite. Then the point $B$ is an upper
boundary point. Since $y_{a}\geq y_{b}$ there exists an upper boundary point
$(x_{a},y)$ with $y\geq y_{b}$ that contradicts the monotonicity of the boundary.
\end{proof}

\begin{corollary}
Any antimatroidal point set $S$ may be represented by its boundary in the
following form:
\[
S=\mathcal{B}_{lower}\vee\mathcal{B}_{upper}=\{X\cup Y:X\in\mathcal{B}%
_{lower},Y\in\mathcal{B}_{upper}\}
\]

\end{corollary}

This result shows that the \textit{convex dimension} of a two-dimensional
poly-antimatroid is two.

\begin{definition}
Convex dimension \cite{Greedoids} $c\dim(S)$ of any antimatroid $S$ is the
minimum number of maximal chains
\[
\emptyset=X_{0}\subset X_{1}\subset...\subset X_{k}=X_{\max}\text{ with }%
X_{i}=X_{i-1}\cup x_{i\text{ }}%
\]
whose union gives the antimatroid $S$.
\end{definition}

The boundary of any poly-antimatroid may be found by using the following
algorithm. Starting with the maximum point $(x_{\max},y_{\max})$ it follows
the upper boundary down to point $(0,0)$ in the first pass, and it follows the
lower boundary in the second pass.

\begin{algorithm}
\label{Alg}\textbf{Upper} \textbf{boundary tracing algorithm}

1. $i:=0$; $x:=x_{\max}$; $y:=y_{\max}$;

2. $B_{i}:=(x,y)$

3. do

\ \ \ \ \ \ \ \ \ \ \ 3.1\ if $(x-1,y)\in S$ then $x:=x-1$, else $y:=y-1$;

\ \ \ \ \ \ \ \ \ \ \ 3.2 $\ i:=i+1$;

\ \ \ \ \ \ \ \ \ \ \ 3.3 \ $B_{i}:=(x,y)$;

\ \ \ \ until $B_{i}=(0,0)$

4. Return the sequence $\mathcal{B}=B_{i},B_{i-1},...,B_{0}$
\end{algorithm}

It is easy to check that Algorithm \ref{Alg} returns a monotone increasing
$N_{4}$-path that only passes over the upper boundary points from
$\mathcal{B}_{upper}$ (for each $i$ point $B_{i}=(x^{i},y^{i})$ has the
maximum $y-$coordinate, i.e., $y^{i}=\underset{y}{\arg\max}\{(x^{i},y)\in
S\}$). Hence, the Upper boundary tracing algorithm returns the upper boundary
of poly-antimatroids.

The Lower boundary tracing algorithm differs from Algorithm \ref{Alg} on
search order only: $(y,x)$ instead of $(x,y)$. Step 3.1 of Algorithm \ref{Alg}
will be as follows:%
\[
3.1\text{ if }(x,y-1)\in S\text{ then }y:=y-1\text{,else }x:=x-1\text{;}%
\]

Correspondingly, the Lower boundary tracing algorithm returns the lower
boundary of poly-antimatroids.

It turned out that a two-dimensional poly-antimatroid known in this section as
an antimatroidal point set is equivalent to special cases of polyominoes or
staircase polygons and is defined by a pair of monotone paths made only with
north and east steps. In the next section we research the path structure of
three-dimensional antimatroidal point sets.

\section{Three-dimensional antimatroidal point set}

In this section we consider a particular case of three-dimensional
antimatroidal point sets.

\begin{definition}
A\ finite non-empty set $C$ of points in digital $3D$ space is a digital
cuboid if
\[
C=\{(x,y,z)\in{\mathbb{Z}}^{3}:x_{\min}\leq x\leq x_{\max}\wedge y_{\min}\leq
y\leq y_{\max}\wedge z_{\min}\leq z\leq z_{\max}\}\text{ and }|C|>1
\]

\end{definition}

A digital cuboid is specified by the coordinates of opposite corners.

Consider the following construction of $n$ cuboids:

\begin{definition}
\label{staircase}The sequence of $n$ cuboids $C_{1},C_{2},...,C_{n}$ is called
regular if \

(a) $x_{\min}^{0}=y_{\min}^{0}=z_{\min}^{0}=0$

(b) $x_{\min}^{i}\leq x_{\min}^{i+1}\wedge y_{\min}^{i}\leq y_{\min}%
^{i+1}\wedge z_{\min}^{i}\leq z_{\min}^{i+1}$

\ \ \ \ and at least one of the inequality is strong \ for each $1\leq i\leq
n-1$

(c) $x_{\min}^{i+1}\leq x_{\max}^{i}\wedge y_{\min}^{i+1}\leq y_{\max}%
^{i}\wedge z_{\min}^{i+1}\leq z_{\max}^{i}$ for each $1\leq i\leq n-1$

(d) $x_{\max}^{i}\leq x_{\max}^{i+1}\wedge y_{\max}^{i}\leq y_{\max}%
^{i+1}\wedge z_{\max}^{i}\leq z_{\max}^{i+1}$ \

\ \ \ \ and at least one of the inequality is strong \ for each $1\leq i\leq
n-1$
\end{definition}

\begin{definition}
The union of elements of regular sequence $\mathcal{C}=C_{1}\cup C_{2}%
\cup...\cup C_{n}$  is called an $n$-step staircase.
\end{definition}

An example of a $2$-step staircase is depicted in Figure $\ref{int_fig3}$.

\begin{figure}[ptbh]
\setlength{\unitlength}{1cm}
\par
\begin{picture}(8,4.5)(-4.5,-0.5)\thicklines
\multiput(0,0.5)(1.0,0){2}{\circle*{0.25}}
\multiput(0,1.5)(1.0,0){2}{\circle*{0.25}}
\multiput(0,2.5)(1.0,0){2}{\circle*{0.25}}
\put(0.7,3.2){\circle*{0.25}}
\multiput(1.4,3.9)(1.0,0){4}{\circle*{0.25}}
\multiput(1.3,0)(1.0,0){3}{\circle*{0.25}}
\multiput(1.3,1)(2,0){2}{\circle*{0.25}}
\multiput(1.3,2)(1.0,0){3}{\circle*{0.25}}
\multiput(2.0,2.7)(2,0){2}{\circle*{0.25}}
\multiput(2.7,3.4)(1.0,0){3}{\circle*{0.25}}
\put(4.0,0.7){\circle*{0.25}}
\multiput(4.7,1.4)(0,1.0){2}{\circle*{0.25}}
\put(0,0.5){\line(1,0){1.3}} \put(0,1.5){\line(1,0){1.3}}
\put(0,0.5){\line(0,1){2.0}} \put(1.3,0){\line(0,1){2.0}}
\put(1.3,0){\line(1,0){2.0}}\put(1.3,2){\line(1,0){2.0}}
\put(3.3,0){\line(0,1){2.0}}\put(4.7,1.4){\line(0,1){2.0}}
\put(2.7,3.4){\line(1,0){2.0}} \put(1.4,3.9){\line(1,0){3.0}}
\put(0,2.5){\line(1,1){1.4}}\put(1.3,2.0){\line(1,1){1.4}}
\put(3.3,2.0){\line(1,1){1.4}}\put(3.3,0.0){\line(1,1){1.4}}
\put(0,2.5){\line(1,0){1.8}} \put(4.4,3.4){\line(0,1){0.5}}
\put(3.9,3.4){\line(1,1){0.5}}
\end{picture}
\caption{$2$-step staircase $C$.}%
\label{int_fig3}%
\end{figure}
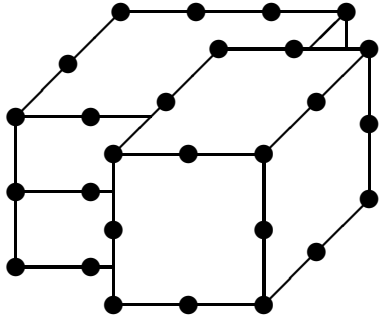

Each point $A=(x_{A},y_{A},z_{A})$ in the digital space ${\mathbb{Z}}^{3}$ may
be considered as a multiset $A$ over $\{x,y,z\}$, where $x_{A}$ is the number
of repetitions of an element $x$, and $y_{A}$ is the number of repetitions of
an element $y$, and $z_{A}$ is the number of repetitions of an element $z$ in
multiset $A$. Then a $1$-step staircase (cuboid) is a poly-antimatroid, since
the family of points is accessible and closed under union ($(x^{1},y^{1}%
,z^{1}),(x^{2},y^{2},z^{2})\in C\Rightarrow(x^{1},y^{1},z^{1})\cup(x^{2}%
,y^{2},z^{2})=(\max(x^{1},x^{2}),\max(y^{1},y^{2}),\max(z^{1},z^{2}))\in C$).

\begin{lemma}
An $n$-step staircase $\mathcal{C}$ is a poly-antimatroid.
\end{lemma}

\begin{proof}
Consider some point $A\neq(0,0,0)\in\mathcal{C}$. Then there exists $i$ such
that $A\in C_{i}$. If $A\neq(x_{\min}^{i},y_{\min}^{i},z_{\min}^{i})$, then,
without lost of generality, $x_{\min}^{i}<x_{A}$, and so $(x_{A}-1,y_{A}%
,z_{A})\in C_{i}\subseteq\mathcal{C}$. If $A=(x_{\min}^{i},y_{\min}%
^{i},z_{\min}^{i})$, then, from Definition \ref{staircase} (b,c) it follows
that $A\in C_{i-1}$and $A\neq(x_{\min}^{i-1},y_{\min}^{i-1},z_{\min}^{i-1})$.
Hence, $(x_{A}-1,y_{A},z_{A})\in C_{i-1}\subseteq\mathcal{C}$ or $(x_{A}%
,y_{A}-1,z_{A})\in C_{i-1}\subseteq\mathcal{C}$ or $(x_{A},y_{A},z_{A}-1)\in
C_{i-1}\subseteq\mathcal{C}$. Thus, an $n$-step staircase $\mathcal{C}$ is accessible.

To prove that $\mathcal{C}$ is closed under union consider two points
$A,B\in\mathcal{C}$. Let $A\in C_{i}$, $B\in C_{j}$, and $i\leq j$. From
Definition \ref{staircase} (b) it follows that $x_{\min}^{j}\leq\max
(x_{A},x_{B})$. From Definition \ref{staircase} (d), $\max(x_{A},x_{B})\leq
x_{\max}^{j}$. So, $A\cup B=(\max(x_{A},x_{B}),\max(y_{A},y_{B}))\in C_{j}$,
i.e., $A\cup B\in\mathcal{C}.$
\end{proof}

Our goal is to find the minimum number of maximal chains which describe an
$n$-step staircase.

First, consider a cuboid $C$ given by two points $(x_{\min},y_{\min},z_{\min
})$ and $(x_{\max},y_{\max},z_{\max})$. We will denote by $P_{X}$ the maximal
chain connecting the points $(x_{\min},y_{\min},z_{\min})$ and $(x_{\max
},y_{\min},z_{\min})$, i.e.,
\[
P_{X}=(x_{\min},y_{\min},z_{\min}),(x_{\min}+1,y_{\min},z_{\min}%
),...,(x_{\max},y_{\min},z_{\min})\text{.}%
\]
$P_{Y}$ and $P_{Z}$ are defined correspondingly. It is easy to see that
\[
C=P_{X}\vee P_{Y}\vee P_{Z}=\{A_{1}\cup A_{2}\cup A_{3}:A_{1}\in P_{X}%
,A_{2}\in P_{Y},A_{3}\in P_{Z}\}
\]

So, if we have three maximal chains $B_{1}$,$B_{2}$,$B_{3}$ from the cuboid
$C$, which pass over all the points of $P_{X}$, $P_{Y}$ and $P_{Z}$, then
$C\subseteq B_{1}\vee B_{2}\vee B_{3}$. On the other hand, since $C$ is closed
under union, we have $B_{1}\vee B_{2}\vee B_{3}\subseteq C$. Thus, $B_{1}\vee
B_{2}\vee B_{3}=C$.

Now, consider the three-dimensional version of Algorithm \ref{Alg}
for search order $(x,y,z)$. Let $(x_{\max},y_{\max},z_{\max})$ be
the maximum point of $n$-step staircase $\mathcal{C}$. The
algorithm builds the chain connecting the maximum point and the
point $(0,0,0)$.

\begin{algorithm}
\label{Alg3}XYZ-Boundary tracing algorithm

1. $i:=0$; $x:=x_{\max}$; $y:=y_{\max}$; $z:=z_{\max}$;

2. $B_{i}:=(x,y,z)$

3. do

\ \ \ \ \ \ \ \ \ \ \ 3.1\ if $(x-1,y,z)\in S$ then $x:=x-1$, else if
$(x,y-1,z)\in S$

\ \ \ \ \ \ \ \ \ \ \ \ \ \ \ \ then $y:=y-1$, else $z:=z-1$;

\ \ \ \ \ \ \ \ \ \ \ 3.2 $\ i:=i+1$;

\ \ \ \ \ \ \ \ \ \ \ 3.3 \ $B_{i}:=(x,y,z)$;

\ \ \ \ until $B_{i}=(0,0,0)$

4. Return the sequence $\mathcal{B}=B_{i},B_{i-1},...,B_{0}$
\end{algorithm}

Repeat Algorithm \ref{Alg3} for search order $(y,z,x)$ and for search order
$(z,x,y)$. As a result, we obtain three monotone increasing $N_{6}$-paths
\cite{Digital} from $(0,0,0)$ to $(x_{\max},y_{\max},z_{\max})$, denoted by
$\mathcal{B}_{Z}$, $\mathcal{B}_{X}$ and $\mathcal{B}_{Y}$, respectively. The
length of any of these chains is equal to $(x_{\max}+y_{\max}+z_{\max})$.

\begin{theorem}
\label{3convex}Any $n$-step staircase $\mathcal{C}=\mathcal{B}_{X}%
\vee\mathcal{B}_{Y}\vee\mathcal{B}_{Z}$.
\end{theorem}

To show this we prove a stronger statement.

Change Algorithm \ref{Alg3} in the following way. Let the XYZ-Boundary tracing
algorithm begin from some point $(x,y,z_{\max})\in\mathcal{C}$ and return
chain $\mathcal{H}_{Z}$; the YZX-Boundary tracing algorithm begins from some
point $(x_{\max},y,z)\in\mathcal{C}$ and returns chain $\mathcal{H}_{X}$; and
the ZXY-Boundary tracing algorithm begins from some point $(x,y_{\max}%
,z)\in\mathcal{C}$ and returns chain $\mathcal{H}_{Y}$.

\begin{lemma}
Any $n$-step staircase $\mathcal{C}=\mathcal{H}_{X}\vee\mathcal{H}_{Y}%
\vee\mathcal{H}_{Z}$.
\end{lemma}

\begin{proof}
Let us proceed by induction on $n$.

For $n=1$ the XYZ-Boundary tracing algorithm beginning from the point
$(x,y,z_{\max})\in\mathcal{C}$ returns chain $\mathcal{H}_{Z}$ that begins
from the point $(0,0,0)$, moves to $(x_{\min},y_{\min},z_{\max})$, then to
$(x_{\min},y,z_{\max})$, and finishes at the point $(x,y,z_{\max})$. So, this
chain passes over all the points of $P_{Z}$. In the same way, the YZX-Boundary
tracing algorithm returns chain $\mathcal{H}_{X}$ that passes over all the
points of $P_{X}$, and the ZXY-Boundary tracing algorithm returns chain
$\mathcal{H}_{Y}$ that passes over all the points of $P_{Y}$. Thus,
$\mathcal{C}=\mathcal{H}_{X}\vee\mathcal{H}_{Y}\vee\mathcal{H}_{Z}$.

Assume that the proposition is correct for all $k<n$ and prove it for $n$. It
is easy to see that the XYZ-Boundary tracing algorithm begins from some point
$(x,y,z_{\max})\in C_{n}$, reaches the point $(x_{\min}^{n},y_{\min}%
^{n},z_{\max}^{n})$, moves down to the point $(x_{\min}^{n},y_{\min}%
^{n},z_{\max}^{n-1})\in C_{n-1}$, and then continues to build chain
$\mathcal{H}_{Z}$. In the same way, the YZX-Boundary tracing algorithm moves
through the point $(x_{\max}^{n-1},y_{\min}^{n},z_{\min}^{n})\in C_{n-1}$, and
the ZXY-Boundary tracing algorithm moves through the point $(x_{\min}%
^{n},y_{\max}^{n-1},z_{\min}^{n})\in C_{n-1}$. The induction assumption
implies $C_{1}\cup C_{2}\cup...\cup C_{n-1}\subseteq\mathcal{H}_{X}%
\vee\mathcal{H}_{Y}\vee\mathcal{H}_{Z}$.

It remains to show that $C_{n}\subseteq\mathcal{H}_{X}\vee\mathcal{H}_{Y}%
\vee\mathcal{H}_{Z}$. To this end we prove that $P_{X}^{n}\subseteq
\mathcal{H}_{X}\vee\mathcal{H}_{Y}\vee\mathcal{H}_{Z}$, $P_{Y}^{n}%
\subseteq\mathcal{H}_{X}\vee\mathcal{H}_{Y}\vee\mathcal{H}_{Z}$, and
$P_{Z}^{n}\subseteq\mathcal{H}_{X}\vee\mathcal{H}_{Y}\vee\mathcal{H}_{Z}$. We
have already seen that the XYZ-Boundary tracing algorithm returns chain
$\mathcal{H}_{Z}$ that covers the part $[(x_{\min}^{n},y_{\min}^{n},z_{\max
}^{n}),(x_{\min}^{n},y_{\min}^{n},z_{\max}^{n-1})]$ from $P_{Z}^{n}$, and the
remaining part $[(x_{\min}^{n},y_{\min}^{n},z_{\max}^{n-1}),(x_{\min}%
^{n},y_{\min}^{n},z_{\min}^{n})]\subseteq C_{1}\cup C_{2}\cup...\cup C_{n-1}$.
So $P_{Z}^{n}\subseteq\mathcal{H}_{X}\vee\mathcal{H}_{Y}\vee\mathcal{H}_{Z}$.
It is easy to verify two other cases.

Finally, $\mathcal{C}=C_{1}\cup C_{2}\cup...\cup C_{n}\subseteq(\mathcal{H}%
_{X}\vee\mathcal{H}_{Y}\vee\mathcal{H}_{Z})\vee(\mathcal{H}_{X}\vee
\mathcal{H}_{Y}\vee\mathcal{H}_{Z})\vee(\mathcal{H}_{X}\vee\mathcal{H}_{Y}%
\vee\mathcal{H}_{Z})=\mathcal{H}_{X}\vee\mathcal{H}_{Y}\vee\mathcal{H}_{Z}$,
since each $\mathcal{H}$ is a chain.

On the other hand, $\mathcal{H}_{X}\vee\mathcal{H}_{Y}\vee\mathcal{H}%
_{Z}\subseteq\mathcal{C}$, and so $\mathcal{C}=\mathcal{H}_{X}\vee
\mathcal{H}_{Y}\vee\mathcal{H}_{Z}$.
\end{proof}

Theorem \ref{3convex} shows that it is enough to know only three maximal
chains $\mathcal{B}_{X}$,$\mathcal{B}_{Y}$, and $\mathcal{B}_{Z}$ to describe
an $n$-step staircase.

\begin{corollary}
An n-step staircase has convex dimension at most $3$.
\end{corollary}

Note that the three points $(0,0,1)$, $(0,1,0)$ and $(1,0,0)$ cannot be
covered by two chains only, since each such point (multiset) cannot be formed
as a union of smaller multisets. So, there exist $n$-step staircases of convex
dimension $3$.

However, the convex dimension of an arbitrary three-dimensional
poly-antimatroid may be arbitrarily large \cite{Eppstein}. Let $S$ be a set of
points:
\[
S=\{(x,y,z):(0\leq x,y\leq N)\wedge(0\leq z\leq1)\wedge(z=1\Rightarrow x+y\geq
N)\}
\]

It is easy to check that $S$ is a three-dimensional poly-antimatroid. Consider
$N+1$ points $(x,y,1)$ with $x+y=N$. Since each of these points cannot be
represented as a union of any points from $S$ with smaller coordinates, the
convex dimension of $S$ is at least $N+1$ \cite{Eppstein}.

In order to characterize the family of three-dimensional poly-antimatroids of
convex dimension $3$, consider a particular case of antimatroids called
\textit{poset antimatroids} \cite{Greedoids}. A poset antimatroid has as its
feasible sets the lower sets of a poset (partially ordered set). The poset
antimatroids can be characterized as the unique antimatroids which are closed
under intersection \cite{Greedoids}.\ We extend this definition to poly-antimatroids.

\begin{definition}
A poly-antimatroid is called a poset poly-antimatroid if it is closed under intersection.
\end{definition}

Now, note that $n$-step staircases are closed under intersection too. Indeed,
consider two points $A,B\in\mathcal{C}$. Let $A\in C_{i}$, $B\in C_{j}$, and
$i\leq j$. From Definition \ref{staircase} (b) it follows that $\min
(x_{A},x_{B})\geq x_{\min}^{i}$. On the other hand, (Definition
\ref{staircase} (d)), $\min(x_{A},x_{B})\leq x_{\max}^{i}$. So, $A\cap
B=(\min(x_{A},x_{B}),\min(y_{A},y_{B}))\in C_{i}$, i.e., $A\cap B\in
\mathcal{C}.$

So, our conjectures are as follows.

\begin{conjecture}
A three-dimensional poset poly-antimatroid is a step staircase.
\end{conjecture}

\begin{corollary}
The convex dimension of a three-dimensional poset poly-antimatroid is at most
$3$.
\end{corollary}

Moreover,

\begin{conjecture}
The convex dimension of an $n$-dimensional poset poly-antimatroid is at most
$n$.
\end{conjecture}

\begin{acknowledgement}
The authors are grateful to the referees for their constructive criticism and
correction of style.
\end{acknowledgement}

\end{document}